\documentclass[11pt]{amsart}

\usepackage[T1]{fontenc}
\usepackage{lmodern}
\usepackage{amsmath,amssymb,amsthm}
\usepackage{mathtools}
\usepackage{mathrsfs}
\usepackage{xcolor}
\usepackage[hidelinks]{hyperref}

\numberwithin{equation}{section}

\theoremstyle{plain}
\newtheorem{thm}{Theorem}[section]
\newtheorem{cor}[thm]{Corollary}
\newtheorem{lem}[thm]{Lemma}
\newtheorem{prop}[thm]{Proposition}

\newtheorem{mainthm}{Theorem}

\newtheorem*{mainconj}{Conjecture}

\theoremstyle{definition}

\theoremstyle{remark}
\newtheorem{rem}[thm]{Remark}

\DeclareMathOperator{\GL}{GL}
\DeclareMathOperator{\SL}{SL}
\DeclareMathOperator{\SO}{SO}

\DeclareMathOperator{\Aut}{Aut}

\DeclareMathOperator{\coh}{H}

\DeclareMathOperator{\pr}{pr}

\DeclareMathOperator{\Stab}{Stab}

\DeclareMathOperator{\Sym}{Sym}

\DeclareMathOperator{\NS}{NS}
\DeclareMathOperator{\Kum}{Kum}
\DeclareMathOperator{\Ort}{O}

\def\bC{{\mathbb{C}}}

\def\bR{{\mathbb{R}}}

\def\bZ{{\mathbb{Z}}}

\def\bP{{\mathbb{P}}}

\def\bQ{{\mathbb{Q}}}

\def\cO{{\mathscr{O}}}
\def\cP{{\mathscr{P}}}

\def\cX{{\mathscr{X}}}

\def\aL{{\mathcal{L}}}

\begin{document}

\title[Oka K3 surfaces]{Periods of Oka K3 surfaces are dense}

\author{Song-Yan Xie}
\address{State Key Laboratory of Mathematical Sciences, Academy of Mathematics and Systems Science, Chinese Academy of Sciences, Beijing 100190, China;  School of Mathematical Sciences, University of Chinese Academy of Sciences, Beijing 100049, China.}
\email{xiesongyan@amss.ac.cn}

\author{Shengyuan Zhao}
\address{Universit\'e Paul Sabatier, Institut de Math\'ematiques de Toulouse, 118, route de Narbonne, F-31062 Toulouse, France}
\email{shengyuan.zhao@math.univ-toulouse.fr}

\begin{abstract}
We give the first known examples of Oka Calabi--Yau manifolds in any dimension: every smooth hypersurface of multidegree $(2,\ldots,2)$ in a product of $\bP^1$ is Oka. Building on this example, we prove that in every local Kuranishi family of K3 surfaces the points with Oka fibers form a dense $G_\delta$ subset of the base, and the sublocus of Oka fibers with vanishing N\'eron--Severi group is likewise a dense $G_\delta$ subset. We also obtain projective and nonprojective Oka Kummer surfaces and prove an analogous density statement in the Kummer period domain. We also prove the existence of Oka Enriques surfaces and Oka irreducible hyperkähler fourfolds.
\end{abstract}

\maketitle

\section{Introduction}
A complex manifold $Y$ is Oka if, for every integer $n\geq 1$, every compact convex set
$C\subset\bC^n$, and every holomorphic map from a neighbourhood of $C$ to $Y$,
the map can be uniformly approximated on $C$ by entire maps $\bC^n\to Y$
\cite{Forstneric2006}. This is the convex approximation property (CAP), which
Forstneri{\v c} proved to be equivalent to several other classical Oka
properties; see also
\cite[Ch.~5]{ForstnericBook}.

The Oka principle has its roots in the work of Oka and Grauert \cite{Oka1939,Grauert1958}, while modern Oka theory began with Gromov's seminal paper~\cite{Gromov1989}. Gromov placed the theory in the framework of the $h$-principle and introduced the notions of elliptic manifolds and dominating sprays. Building on Gromov's framework, Forstneri{\v c} introduced subellipticity and proved that subelliptic manifolds are Oka \cite{ForstnericSub}. His work on CAP and its variants led to the modern formulation of the Oka property \cite{Forstneric2006,ForstnericBook}. Despite these general criteria, deciding
whether a given complex manifold is Oka remains difficult. A central problem in
Oka theory is to produce new examples of Oka manifolds and to identify concrete geometric
structures that imply the Oka property.

This problem is already subtle for compact complex surfaces. Forstneri{\v c} and L{\'a}russon made a systematic study of their holomorphic flexibility
properties~\cite{ForstnericLarusson2014}. Among the compact Oka surfaces known from their analysis are rational surfaces, ruled surfaces over $\bP^1$ or an
elliptic curve, complex tori, bielliptic surfaces, Kodaira surfaces, minimal Hopf and Enoki surfaces \cite[Th.~4 and \S~1]{ForstnericLarusson2014}. The Oka problem for K3 surfaces had already been raised in
\cite[Problem~C]{ForstnericLarusson2011} where it is proved that every Kummer K3
surface is strongly dominable, but it left open whether any Kummer surface, or any K3 surface, is Oka \cite[Cor.~3 and \S~1]{ForstnericLarusson2014}. More recently, Kummer surfaces
and elliptic K3 surfaces were shown to satisfy the weaker Oka-$1$ property, while their Oka property remained open  
\cite[Prop.~8.4, Cor.~8.6, and \S~8.2]{AlarconForstneric2025}.  Prior to the present work, no example of an Oka K3 surface was known.

We fill this gap and show that Oka K3 surfaces are abundant. By a local Kuranishi family $\pi:\mathcal X\to B$ of a K3 surface $X_0$ we mean a proper
holomorphic deformation with central fiber $X_0$, where $B$ is a smooth complex manifold, such that the Kodaira--Spencer map
$T_0B\to \coh^1(X_0,T_{X_0})$ is an isomorphism. 
Recall that a $G_\delta$ subset of a topological space is a countable intersection of open subsets.

\begin{mainthm}\label{mainthm:K3-Kuranishi}
Let $\pi:\mathcal X\to B$ be a local Kuranishi family of a K3 surface with fibers $X_b,b\in B$. Then both 
\[
 B_{\mathrm{Oka}}=\{b\in B\mid X_b\text{ is Oka}\}, \;{and}\; \{b\in B_{\mathrm{Oka}}\mid \NS(X_b)=0\}
\]
are dense $G_\delta$ subsets of $B$.
\end{mainthm}

In other words, for every K3 surface, a Baire-generic member of its local deformation space is Oka. Here ``generic'' is understood in the sense of Baire category, not in the Zariski sense. Since a K3 surface with trivial N\'eron--Severi group is nonprojective and contains no curves, we obtain especially a dense $G_\delta$ set of nonprojective, curve-free Oka K3 surfaces in every local deformation space. 

The density in Theorem~\ref{mainthm:K3-Kuranishi}, together with the variety of examples produced below and in a forthcoming paper of ours, leads us to the following conjecture.

\begin{mainconj}
Every K3 surface is Oka.
\end{mainconj}

For $m\geq 2$, a smooth hypersurface of multidegree $(2,\ldots,2)$ in $(\bP^1)^{m+1}$ is a Calabi--Yau manifold, called a \emph{Wehler Calabi--Yau $m$-fold} (see \cite{CantatOguiso}). When $m=2$ it is a K3 surface. 
The geometric starting point of the proof of Theorem~\ref{mainthm:K3-Kuranishi} is the following explicit families of projective Oka Calabi--Yau manifolds.

\begin{mainthm}[{Cor.~\ref{cor:222}}]\label{mainthm:222}
Every Wehler Calabi--Yau manifold is Oka. 
\end{mainthm} 

By using Wehler K3 surfaces, we can also prove
\begin{mainthm}[{Cor.~\ref{cor:unnodal-enriques}}]\label{mainthm:Enriques}
Every unnodal Enriques surface is Oka.
\end{mainthm}
and
\begin{mainthm}[{Thm.~\ref{thm:oka-fourfold}}]\label{mainthm:hyperkahler}
There exist irreducible hyperkähler fourfolds which are Oka.
\end{mainthm}
We note that in any moduli space of polarized Enriques surfaces, the unnodal ones form a dense Zariski-open locus. We don't know whether any or every nodal Enriques surface is Oka or not.

Theorem~\ref{mainthm:K3-Kuranishi} can also be formulated in Hodge-theoretic terms. For background on markings, the K3
lattice, and the period domain, see~\cite{Huybrechts}. A marking of a K3 surface $X$ is an isometry $\varphi:\coh^2(X,\bZ)\to\Lambda$, where
$\Lambda=3U\oplus2E_8(-1)$ is the K3 lattice, $U$ has Gram matrix
$\left(\begin{smallmatrix}0&1\\1&0\end{smallmatrix}\right)$, and $E_8(-1)$ is
the negative definite $E_8$ lattice. If
$\omega\in \coh^{2,0}(X)$ is a non-zero $2$-form, then the \emph{period} is defined to be the oriented positive plane
$P_X=\varphi(\operatorname{span}_{\bR}\{\operatorname{Re}\omega,
\operatorname{Im}\omega\})$.  The \emph{K3 period domain} $\mathcal D$ is defined to be the space of oriented two-planes on which the form of $\Lambda_\bR$ is positive definite. We denote by $\NS(X)=\coh^2(X,\bZ)\cap \coh^{1,1}(X)$ the N\'eron--Severi lattice. For a marking $\varphi$, we have $\varphi(\NS(X))=\Lambda\cap P_X^\perp$. 
For a complex two-torus $A$, we denote by $\Kum(A)$ the minimal resolution of $A/\{\pm 1\}$. Such a surface is K3 and is called a \emph{Kummer surface}.  Let $A[2]$ be the subgroup of $2$-torsion points. The quotient $A/\{\pm 1\}$ has sixteen nodes. If $E_a$, $a\in A[2]$, are the exceptional curves (of the resolution) on $\Kum(A)$, then the lattice
\[
 \Kum:=\coh^2(\Kum(A),\bZ)\cap
 \operatorname{span}_\bQ\{[E_a]\mid a\in A[2]\}
\]
is called the \emph{Kummer lattice}. It is negative definite of rank $16$ and has discriminant $2^6$. Its primitive embedding in $\Lambda$ is unique up to composition with elements of the orthogonal group $\Ort(\Lambda)$ \cite[Ch.~14, Cor.~3.15]{Huybrechts}. Fix such an embedding $\Kum\subset\Lambda$, let $L:=\Kum^\perp$, and let
$\mathcal D_K:=\{P\in\mathcal D\mid P\perp \Kum\}$. Then $L$ has signature $(3,3)$, $\dim_\bC\mathcal D=20$, and $\dim_\bC\mathcal D_K=4$.  Nikulin's criterion~\cite[Th.~3]{NikulinKummer} says that a K3 surface is Kummer if and only if its N\'eron--Severi lattice
contains a primitive copy of $\Kum$. Hence every period in $\mathcal D_K$ represents a marking of a Kummer surface.

\begin{mainthm}\label{mainthm:period}
Let $\mathcal O\subset\mathcal D$ be the set of periods represented by marked K3 surfaces which are Oka, and let $\mathcal O_K:=\mathcal O\cap\mathcal D_K$.
\begin{enumerate}
\item The set $\mathcal O$ is dense $G_\delta$ in $\mathcal D$, and so is
$\{P\in\mathcal O\mid \Lambda\cap P^\perp=0\}$.
\item The set $\mathcal O_K$ is dense $G_\delta$ in $\mathcal D_K$, and so
is $\{P\in\mathcal O_K\mid \Lambda\cap P^\perp=\Kum\}$.
\end{enumerate}
In particular, there exist projective and nonprojective Oka Kummer surfaces.
\end{mainthm}

The construction leading to the
theorem also produces a positive-dimensional family of projective Oka Kummer surfaces. In a forthcoming paper, we will prove that every projective Kummer surface is Oka. We note that the density of Oka parameters asserted by Theorem~\ref{mainthm:period} works on the period domain which is a Hodge-theoretic parameter space, and it does not imply density of Oka parameters in a moduli of polarized K3 surfaces.

\section{Strategy and plan}
In Section~\ref{sec:symplectic-sprays} we prove Theoerem~\ref{mainthm:222}. The idea is as follows. Whenever a Calabi--Yau manifold admits a one-dimensional fibration, or whenever a hyperkähler manifold admits a Lagrangian fibration, one can use the holomorphic volume form or the  holomorphic sympletic form to build a holomorphic spray. A Wehler Calabi--Yau manifold admits naturally fibrations onto products of $\bP^1$ for which the corresponding one-dimensional sprays form a dominating family. This construction could also work for some other K3 surfaces with two elliptic fibrations, especially those with rich dynamics, see Cor.~\ref{cor:non-elementary-k3}. The method is flexible and could also be adapted to some other situations. Theorem~\ref{mainthm:Enriques} about Oka Enriques sufaces is a quick consequence. However the construction of the hyperkähler example is technically more involved and is thus postponed to Section~\ref{sec:oka-fourfold}, though the mechanism of the construction is the same as those in Section~\ref{sec:symplectic-sprays}.

In Section~\ref{sec:fermat} we discuss one very special example of K3 surface: the Fermat quartic. Drawing on classical algebraic geometry of pencils of hyperplane sections, we show that it is simultaneously a K3 surface of type $(2,2,2)$ and a Kummer surface. Therefore Theorem~\ref{mainthm:222} implies the existence of an Oka Kummer surface. Such classical algebro-geometric arguments also play a crucial rule in the construction of open Oka varieties by the first author~\cite{Xie}. We will also show that the particular needed feature of the Fermat quartic survives in a positive-dimensional deformation family. This deformation is crucial for our later application. 

In Section~\ref{sec:periods} we apply Ratner's theory to obtain the density assertion in Theorem~\ref{mainthm:K3-Kuranishi}, while the $G_\delta$ property is already available by L\'arusson's work~\cite[Cor.~8]{Larusson}. Ratner~\cite{Ratner} established deep classification theorems for orbit closures of unipotent flows on homogeneous spaces. Verbitsky~\cite{Verbitsky1,Verbitsky} applied this theory to the action of the K3 modular group on the period domain. In particular Verbitsky asserts that a totally irrational K3 period has dense orbit in the period domain. 
Our deformation of the Fermat quartic constructed in Section~\ref{sec:fermat} provides Oka marked K3 surfaces with such  totally irrational periods.

\section{Holomorphic forms, fibrations and sprays}\label{sec:symplectic-sprays}
Gromov~\cite{Gromov1989} introduced the notion of spray. A \emph{spray} on a complex manifold $X$ is a holomorphic vector bundle $p:E\to X$ and a holomorphic map $s:E\to X$ such that $s(0_x)=x$. Its
vertical differential at $0_x$ is the restriction of $ds_{0_x}$ to the fiber direction $E_x$. A finite set of sprays is dominating if the images of these vertical differentials span $T_xX$ at every point. A manifold carrying a dominating family is called \emph{subelliptic}. It is known that every subelliptic manifold is Oka (see \cite{ForstnericSub} and \cite[Cor.~5.6.14]{ForstnericBook}). We say that a spray $s:E\to X$ with bundle projection $p:E\to X$ \emph{preserves} a fibration $f:X\to B$ if $f\circ s=f\circ p$. In this case $s(E_x)$ is contained in the fiber $f^{-1}(f(x))$ for every $x$.

\subsection{Vertical integration}\label{subsec:vertical-integration}
Let $f:X\to B$ be a holomorphic map between complex manifolds and let $\pi:E\to B$ be a holomorphic vector bundle. A point of $f^*E$ is a pair $(x,\lambda)$ with $\lambda\in E_{f(x)}$. 

We will say that a spray $s:f^*E\to X$ is \emph{additive over $B$} if it preserves $f$ and if
\begin{equation}\label{eq:additive-spray}
s(s(x,\lambda),\mu)=s(x,\lambda+\mu), \quad \lambda,\mu\in E_{f(x)}.
\end{equation}
For a local section $a$ of an additive $E$, we define 
\begin{equation}\label{eq:additive-biholomorphism}
\tau_a(x):=s(x,a(f(x))).
\end{equation}
By~\eqref{eq:additive-spray}, this is a local biholomorphism with inverse $\tau_{-a}$.

Assume from now on that $A:f^*E\to T_X$ is a holomorphic morphism of vector bundles such that $df\circ A=0$. For a local section $a$ of $E$, we define its \emph{vertical vector field} by $V_a(x):=A_x(a(f(x)))$.

\begin{lem}\label{lem:vertical-spray}
There is a spray on $f^*E$ which preserves $f$ and has vertical differential $A$. It is additive if $[V_a,V_b]=0$ for all local sections $a,b$ defined on a common open set. This condition always holds when $E$ has rank one.
\end{lem}
\begin{proof}
We consider $f^*E$ as a complex submanifold of $X\times E$. A vector
$(v,\eta)\in T_xX\times T_\lambda E$ is tangent to $f^*E$ if and
only if $df_x(v)=d\pi_\lambda(\eta)$. Define locally a vector field on $f^*E$ by $\widetilde V_{(x,\lambda)}:=(A_x(\lambda),0)$, where $0$ is the
zero vector in $T_\lambda E$. Since $df\circ A=0$, the local definitions glue to a globally defined holomorphic vector field $\widetilde V$ on $f^*E$.

The flow of $\widetilde V$ through $(x,\lambda)$ remains in the compact set $f^{-1}(f(x))\times\{\lambda\}$ and exists for every complex time. Its time-one map, composed with the projection onto $X$, gives rise to a spray $s$ preserving $f$. At $\lambda=0$ the vector field vanishes, and its first variation in the parameter is $A_x$. Thus the vertical differential of the spray $s$ is $A$.

If the $V_a$ commute, then the flow of a sum is the composition of their flows, and we obtain~\eqref{eq:additive-spray}. In the rank-one case, two such vector fields have the form $(u\circ f)V$ and $(v\circ f)V$, where $u,v$ are holomorphic functions on the base and $V$ is vertical. Since $V$ annihilates both coefficients, their bracket vanishes. 
\end{proof}

\subsection{Holomorphic symplectic forms}\label{subsec:symplectic-forms}
A \emph{holomorphic symplectic form} is a closed, everywhere nondegenerate holomorphic $2$-form. Now we consider a holomorphic symplectic manifold $X$ equipped with a such a form $\sigma$.

For a tangent vector $v$ and a differential $k$-form $\eta$, we denote by
$\iota_v\eta$ the contraction of $\eta$ with $v$, i.e.\ 
$(\iota_v\eta)(v_1,\ldots,v_{k-1})=\eta(v,v_1,\ldots,v_{k-1})$. We denote the inverse of the contraction $v\mapsto\iota_v\sigma$ by $\sigma^\sharp:T_X^*\to T_X$. For a subspace $W\subset T_xX$, its \emph{symplectic orthogonal} is defined as $W^{\perp}=\{v\in T_xX \mid\sigma_x(v,w)=0\text{ for all }w\in W\}$. A smooth submanifold $Y\subset X$ is called \emph{coisotropic} if $(T_xY)^{\perp}\subset T_xY$ at every $x\in Y$, and \emph{Lagrangian} if equality holds. In particular, smooth hypersurfaces are coisotropic. A fibration on $X$ is called \emph{Lagrangian} if its smooth fibers are Lagrangian.

If $f:X\to B$ is a proper surjective holomorphic map between complex manifolds, and if the smooth fibers of $f$ are coisotropic, then we can define a morphism of bunldes $A_f:f^*T_B^*\to T_X$ by
\begin{equation}\label{eq:hamiltonian-differential}
 A_f:=\sigma^\sharp\circ df^*,\qquad
 \sigma_x(A_{f,x}(\lambda),v)=\lambda(df_x(v)).
\end{equation}
Let $f_j:X\to B_j$, $1\leq j\leq k$, be finitely many maps satisfying the above hypotheses. 

\begin{prop}\label{prop:finite-fibrations}
The morphism $A_f$ is the vertical differential of an additive spray on $T_B^*$. If $X$ is compact and $(f_1,\ldots,f_k)$ is an immersion, then the corresponding sprays form a dominating family, hence $X$ is Oka.
\end{prop}

\begin{proof}
Over the regular values of $f$, we have $\operatorname{im}A_f=(\ker df)^{\perp}\subset\ker df$. The regular locus being dense and open, we have $df\circ A_f=0$ everywhere. 

Let $a,b$ be holomorphic $1$-forms on an open set $U\subset B$. On $f^{-1}(U)$, recall that we have vector fields $V_a(x)=A_{f,x}(a(f(x)))$ and
$V_b(x)=A_{f,x}(b(f(x)))$. By~\eqref{eq:hamiltonian-differential}, we have
$\iota_{V_a}\sigma=f^*a$ and $\iota_{V_b}\sigma=f^*b$. Since $df(V_a)=df(V_b)=0$, for every holomorphic differential form $\eta$ of positive degree on $U$ we have $\iota_{V_a}(f^*\eta)=\iota_{V_b}(f^*\eta)=0$. 

We denote by $\mathcal L_V$ the Lie derivative along a vector
field $V$. By Cartan's formula and $d\sigma=0$, we obtain $\mathcal L_{V_a}\sigma=d(\iota_{V_a}\sigma)=f^*(da)$. Moreover, $\mathcal L_{V_a}(f^*b)=d(\iota_{V_a}(f^*b))+\iota_{V_a}(f^*(db))=0$. We get thus
\[
 \iota_{[V_a,V_b]}\sigma =\mathcal L_{V_a}(\iota_{V_b}\sigma)-\iota_{V_b}(\mathcal L_{V_a}\sigma)=-\iota_{V_b}(f^*(da))=0.
\]
As $\sigma$ is nondegenerate, we have $[V_a,V_b]=0$. By Lemma~\ref{lem:vertical-spray}, we obtain the desired additive spray.

If the product map is an immersion, its dual differential is surjective. Thus the images of the $df_j^*$ span $T_x^*X$. By~\eqref{eq:hamiltonian-differential}, the vertical differentials span $T_xX$ at every point. In other words these sprays form a finite dominating family. Hence $X$ is
subelliptic and therefore Oka by Forstneri{\v c}'s theorem~\cite{ForstnericSub}.
\end{proof}

\subsection{Holomorphic volume forms}\label{subsec:volume-forms}
Let now $X$ have dimension $m\geq 2$ and admit a nowhere-vanishing holomorphic $m$-form $\Omega$, i.e.\ a holomorphic volume form. For a proper holomorphic map $f:X\to B$ to a smooth $(m-1)$-fold, define a morphism of bundles $A_f:f^*K_B\to T_X$, where $K_B$ denotes the canonical bundle of $B$, by
\begin{equation}\label{eq:volume-differential}
 \iota_{A_{f,x}(\lambda)}\Omega_x 
 :=\bigl(\textstyle\bigwedge^{m-1}(df_x)^*\bigr)(\lambda).
\end{equation}
For a holomorphic map $\Phi:X\to B_1\times\cdots\times B_N$ to a product of smooth curves and a subset $J\subset\{1,\ldots,N\}$ of cardinality $m-1$, we denote by $f_J:X\to\prod_{j\in J}B_j$ the composition of $\Phi$ with the coordinate projection.

\begin{prop}\label{prop:volume-product-curves}
The morphism $A_f$ is the vertical differential of an additive spray on $f^*K_B$ over $X$. If $X$ is compact and $\Phi$ is an immersion, then the sprays associated with the $f_J$ form a finite dominating family, hence $X$ is Oka.
\end{prop}

\begin{proof}
For $\xi\in T^*_{f(x)}B$, we obtain from~\eqref{eq:volume-differential} that $f^*\xi\wedge\iota_{A_f(\lambda)}\Omega=f^*(\xi\wedge\lambda)=0$. Since the left-hand side is $\xi(df(A_f(\lambda)))\Omega$, we deduce that $df\circ A_f=0$. Thus we can construct a spray preserving $f$ by Lem.~\ref{lem:vertical-spray}. The additive property of the spray follows from the rank-one case of Lem.~\ref{lem:vertical-spray}.

If $X$ is compact, every $f_J$ is proper. If $\Phi$ is also an immersion, the images of the cotangent lines of the factors span $T_x^*X$. Their $(m-1)$-fold exterior products therefore span $\bigwedge^{m-1}T_x^*X$. By~\eqref{eq:volume-differential}, the vertical differentials span $T_xX$. As in the proof of Prop.~\ref{prop:finite-fibrations} we conclude by Forstneri{\v c}'s theorem~\cite{ForstnericSub}.
\end{proof}

Recall that, for $m\geq 2$, a \emph{Wehler Calabi--Yau $m$-fold} is a smooth hypersurface of multidegree $(2,\ldots,2)$ in $(\bP^1)^{m+1}$; the fact that it is Calabi--Yau can be seen easily by the adjunction formula (see \cite{CantatOguiso}). 

\begin{cor}\label{cor:222}
Every Wehler Calabi--Yau manifold is Oka. 
\end{cor}

\begin{proof}
A Calabi--Yau manifold admits a holomorphic volume form. The inclusion of a Wehler manifold into $(\bP^1)^{m+1}$ satisfies the hypotheses of Prop.~\ref{prop:volume-product-curves}. Remark that when $m=2$ the volume form is a symplectic form and we can also apply Prop.~\ref{prop:finite-fibrations} to the three projections onto $\bP^1$.
\end{proof}

We say that a complex Enriques surface is \emph{unnodal} if it contains no $(-2)$-curves. 

\begin{cor}\label{cor:unnodal-enriques}
Every unnodal Enriques surface is Oka.
\end{cor}

\begin{proof}
By~\cite[Th.~5.1]{BorisovNuer}, every unnodal Enriques surface $Y$ has an étale double cover $X$ which is a Wehler K3 surface. By Cor.~\ref{cor:222} $X$ is Oka. We conclude because the Oka property survives under topological coverings, see~\cite[Prop.~5.6.3]{ForstnericBook}. 
\end{proof}

\subsection{A dynamical criterion}\label{subsec:dynamical-criterion}
Let $X$ be a complex manifold. For each index $i$, let $f_i:X\to B_i$ be a proper holomorphic map, let $E_i\to B_i$ be a holomorphic vector bundle, and let $s_i:f_i^*E_i\to X$ be an additive spray, with vertical differential $A_i$. Define the set
\begin{equation}\label{eq:domination-locus}
 U:=\left\{x\in X \mid \sum_i\operatorname{im}(A_{i,x})=T_xX\right\}.
\end{equation}
The set $U$ is open because the spanning condition at one point involves only finitely many vectors which persist in a neighborhood of that point.

For a holomorphic section $a$ of $E_i$ over an open subset of $B_i$, Recall that we have defined in \eqref{eq:additive-biholomorphism} a local biholomorphism $\tau_{i,a}(x)=s_i(x,a(f_i(x)))$. Let $\cP$ be a pseudogroup generated by such local biholomorphisms and the group $\Aut(X)$. The $\cP$-orbit of a point $x$ is the set of points obtainable from $x$ by finite compositions of the generators and their inverses, wherever these compositions are defined.

Following \cite{Kusakabelocal}, for a holomorphic map $h:S\to X$, we will say that a \emph{spray over $h$} with parameter space $\bC^n$ is a holomorphic map $H:S\times\bC^n\to X$ satisfying $H(s,0)=h(s)$ for every $s\in S$. We call $d_tH(s,0):\bC^n\to T_{h(s)}X$ its parameter differential. 

\begin{prop}\label{prop:orbit}
If every $\cP$-orbit intersects $U$, then $X$ is Oka.
\end{prop}

\begin{proof}
By a theorem of Kusakabe~\cite[Cor.~4.1]{Kusakabelocal}, it suffices to prove that, for every holomorphic map $h:S\to X$ from a Stein manifold $S$ and every $s_0\in S$, there are finitely many sprays over $h$ whose parameter differentials at $s_0$ span $T_{h(s_0)}X$.

For each $i$, by Cartan's theorem~A~\cite{Cartan} we can find finitely many global sections $p_{i,j},j=1,\ldots,\nu_i$ of $(f_i\circ h)^*E_i$ spanning its fiber at $s_0$. The map 
\[
K_i(s,t):=s_i\Bigl(h(s),\sum_j t_j p_{i,j}(s)\Bigr), \quad s\in S,\; t\in \bC^{\nu_i}
\]
defines a spray over $h$ and satisfies $\operatorname{im}\bigl(d_tK_i(s_0,0)\bigr)=\operatorname{im}(A_{i,h(s_0)})$. If $h(s_0)\in U$, then by definition of $U$ we deduce that finitely many of these sprays over $h$ satisfy the above Kusakabe's conditions.

We prove the assertion simultaneously for all $S,h,s_0$ by induction on the length of a composition of generators carrying $h(s_0)$ into $U$. The case of length zero is proved. If the first generator is an automorphism $g$, then we can apply induction to $g\circ h$ and compose the corresponding sprays by $g^{-1}$. Assume now that $\gamma_l\circ\ldots \circ\gamma_1(h(s_0))\in U$, and that $\gamma_1$ is a local biholomorphism $\tau_{i,a}$ where $a$ is a local section of $E_i\to B$. By Cartan's theorem~A, there exists a global section $p$ of $(f_i\circ h)^*E_i$ with
\begin{equation}\label{eq:section-p}
p(s_0)=a\bigl(f_i(h(s_0))\bigr).
\end{equation}
Then we define $h_1(s):=s_i(h(s),p(s))$. We have
$h_1(s_0)=\gamma_1(h(s_0))$, so $\gamma_l\circ\cdots\circ\gamma_2$ sends $h_1(s_0)$ into $U$. By induction, there are finitely many sprays $H_k:S\times\bC^{n_k}\to X$ over $h_1$ whose parameter differentials at $s_0$ span $T_{h_1(s_0)}X$.

Since $f_i\circ h_1=f_i\circ h$, the restriction of $(f_i\circ H_k)^*E_i$ to $S\times\{0\}$ coincides with $(f_i\circ h)^*E_i$. By Cartan's theorem~B~\cite{Cartan} we can extend $(-p)$ to a global section $q_k$ of $(f_i\circ H_k)^*E_i$ on $S\times\bC^{n_k}$. Define $\widehat H_k:S\times\bC^{n_k}\to X$ by
\begin{equation}\label{eq:recenter-spray}
 \widehat H_k(s,t)=s_i(H_k(s,t),e),
 \;\text{where }q_k(s,t)=((s,t),e), \, e\in(E_i)_{f_i(H_k(s,t))}.
\end{equation}
If $p(s)=(s,e)$, then $h_1(s)=s_i(h(s),e)$ and $q_k(s,0)=((s,0),-e)$. By~\eqref{eq:additive-spray} and~\eqref{eq:recenter-spray}, we obtain $\widehat H_k(s,0)=s_i(h_1(s),-e)=s_i(s_i(h(s),e),-e)=h(s)$. Hence $\widehat H_k$ is a spray over $h$.

Let $\pi:f_i^*E_i\to X$ be the projection. Define
$\widetilde H_k:S\times\bC^{n_k}\to f_i^*E_i$ by
$\widetilde H_k(s,t)=(H_k(s,t),e)$, where
$q_k(s,t)=((s,t),e)$. Then we have $\pi\circ\widetilde H_k=H_k$ and, by~\eqref{eq:recenter-spray}, $s_i\circ\widetilde H_k=\widehat H_k$.

Let $\xi=\bigl(h_1(s_0),-a(f_i(h_1(s_0)))\bigr)\in f_i^*E_i$. Since $q_k|_{S\times\{0\}}=-p$ and $f_i\circ h_1=f_i\circ h$, by~\eqref{eq:section-p} we obtain that $\widetilde H_k(s_0,0)=\xi$ for every $k$.
By induction, the images of $d_tH_k(s_0,0)$ span $T_{h_1(s_0)}X$. As $d\pi|_\xi\circ d_t\widetilde H_k(s_0,0)=d_tH_k(s_0,0)$,
we deduce that
\begin{equation}\label{eq:recenter-lift-span}
 \ker(d\pi|_\xi)
 +\sum_k\operatorname{im}(d_t\widetilde H_k(s_0,0))
 =T_\xi(f_i^*E_i).
\end{equation}

For $x$ near $h_1(s_0)$, we have $s_i(x,-a(f_i(x)))=\tau_{i,-a}(x)$.
Since $\tau_{i,-a}$ is a local biholomorphism, by the chain rule
we deduce that $ds_i|_\xi$ is surjective. Moreover, by~\eqref{eq:additive-spray} and~\eqref{eq:section-p}, for every $u\in(E_i)_{f_i(h(s_0))}$ we have
$s_i(h_1(s_0),-a(f_i(h_1(s_0)))+u)=s_i(h(s_0),u)$. As $\ker(d\pi|_\xi)\simeq(E_i)_{f_i(h(s_0))}$, the restriction of $ds_i|_\xi$ to $\ker(d\pi|_\xi)$ is just $A_{i,h(s_0)}$. Applying $ds_i|_\xi$ to~\eqref{eq:recenter-lift-span}, we therefore obtain
\begin{equation}\label{eq:recenter-differential}
 \operatorname{im}(A_{i,h(s_0)})
 +\sum_k\operatorname{im}(d_t\widehat H_k(s_0,0))
 =T_{h(s_0)}X.
\end{equation}
Since $\operatorname{im}(d_tK_i(s_0,0))=\operatorname{im}(A_{i,h(s_0)})$,
by~\eqref{eq:recenter-differential} we conclude that the parameter
differentials of $K_i$ and the $\widehat H_k$ span $T_{h(s_0)}X$.
This completes the induction, and hence $X$ is Oka.
\end{proof}

\begin{cor}\label{cor:double-cover-k3}
Let $\pi:X\to\bP^1\times\bP^1$ be the double cover branched over a smooth curve of bidegree $(4,4)$. The surface $X$ is an Oka K3 surface.
\end{cor}

\begin{proof}
The double cover of $\bP^1\times\bP^1$ branched over a smooth curve of bidegree $(4,4)$ is a K3 surface, by a quick application of the adjunction formula (see \cite[Th.~2.2]{Reid}). We will fix a holomorphic symplectic form on $X$. 

Let $f_i=\pr_i\circ\pi$, where $\pr_i:\bP^1\times\bP^1\to\bP^1$ is the $i$-th projection. By Prop.~\ref{prop:finite-fibrations} we obtain additive sprays $s_i:f_i^*T_{\bP^1}^*\to X$ with vertical differentials $A_i$. Let $\cP$ be the pseudogroup defined above with respect to these sprays.

By~\eqref{eq:hamiltonian-differential}, the images of $A_{1,x}$ and $A_{2,x}$ span $T_xX$ exactly when $d\pi_x$ is invertible. Let $R\subset X$ be the ramification curve of $\pi$. Thus the set $U$ defined in~\eqref{eq:domination-locus} is $X\setminus R$.

Fix $x\in R$. As $R$ is smooth, in suitable local coordinates centered at $x$ and $\pi(x)$, we have $\pi(z,w)=(z,w^2)$ and $R=\{w=0\}$. Hence $\ker(d\pi_x)\cap T_xR=\{0\}$. Let $i\in\{1,2\}$ such that $df_{i,x}\neq 0$. Remark that $\operatorname{im}(A_{i,x})=\ker(df_{i,x})=\ker(d\pi_x)$. For a local holomorphic one-form $a$ on $\bP^1$ with $a(f_i(x))\ne0$, we have
\[
 \left.\frac{d}{dt}\right\vert_{t=0}\tau_{i,ta}(x)
 =A_{i,x}\bigl(a(f_i(x))\bigr)\notin T_xR.
\]
As a consequence, $\tau_{i,ta}(x)\in U$ for every sufficiently small $t\neq 0$. Thus every $\cP$-orbit meets $U$, and we conclude by Prop.~\ref{prop:orbit}.
\end{proof}

Let $X$ be a projective K3 surface. We say that $\Aut(X)$ is \emph{non-elementary} if its action on $\GL(\coh^2(X,\bZ))$ contains a nonabelian free subgroup. An automorphism $g\in\Aut(X)$ is called \emph{parabolic} if it has infinite order and preserves a genus-one fibration. 

\begin{cor}\label{cor:non-elementary-k3}
Assume that $X$ is a K3 surface such that $\Aut(X)$ is non-elementary,  contains a parabolic automorphism, and preserves no nonempty proper algebraic subset of $X$. Then $X$ is Oka.
\end{cor}

\begin{proof}
As $\Aut(X)$ is non-elementary, there are two parabolic automorphisms whose invariant genus-one fibrations $f_1,f_2:X\to\bP^1$ are distinct (see \cite{Cantat2001K3Group}). We will fix a holomorphic symplectic form $\sigma$ on $X$. By Prop.~\ref{prop:finite-fibrations}, we have additive sprays $s_i:f_i^*T_{\bP^1}^*\to X$ with vertical differentials $A_i$. Let $\cP$ be the pseudogroup considered above generated by $\Aut(X)$ and the local biholomorphisms defined by the sprays. 

The set $U$ defined in~\eqref{eq:domination-locus} satisfies
\begin{equation}\label{eq:non-elementary-domination}
 U=\left\{x\in X\mid \ker(df_{1,x})\cap\ker(df_{2,x})=\{0\}\right\}.
\end{equation}
As the morphism $(f_1,f_2):X\to\bP^1\times\bP^1$ is generically
finite, $U$ is a nonempty Zariski-open subset of $X$.

For every $x\in X$, the Zariski closure of the orbit $\Aut(X)\cdot x$ is a
nonempty $\Aut(X)$-invariant algebraic subset. By hypothesis it is $X$, so the orbit intersects $U$. We conclude by Prop.~\ref{prop:orbit}. 
\end{proof}

\begin{rem}
Cor.~\ref{cor:non-elementary-k3} is one of the major inspiration/motivation for our work. We refer to \cite{CantatDujardin2023Random} for examples of K3 surfaces satisfying the hypothesis of Cor.~\ref{cor:non-elementary-k3}.
\end{rem}

\section{The Fermat quartic and Kummer surfaces}\label{sec:fermat}
In this section we will show that many Kummer surfaces are Oka, including some examples with totally irrational period. In the first subsection we study the Fermat quartic and in the second subsection we construct one deformation family of it.

\subsection{A \texorpdfstring{$(2,2,2)$}{(2,2,2)} model}\label{subsec:222model}

Choose a root of unity $\zeta\in\bC$ with $i=\zeta^2$. Consider the Fermat quartic in $\bP^3$ defined in homogeneous coordinates by 
$x_0^4+x_1^4+x_2^4+x_3^4=0$. It is a K3 surface, see \cite[\S~1.1.3]{Huybrechts}. Via the linear change of coordinates
$x_0=\zeta X$, $x_1=Y$, $x_2=\zeta Z$, $x_3=W$, it is isomorphic to 
\[
Q=\{Y^4+W^4=X^4+Z^4\}\subset\bP^3.
\]
It contains the three pairwise disjoint lines
\[
L_1\!=\{X\!=\!Y, Z\!=\!W\},
L_2\!=\{X\!=\!iY, Z\!=\!-W\},
L_3\!=\{X\!=\!-Y, Z\!=\!iW\}.
\]

Let $H$ denote the hyperplane class on $Q$. 
For each $j$, every plane through $L_j$ cuts $Q$ in a divisor of the form $L_j+C$, where $C$ is a plane cubic. As the plane varies, these cubics form a pencil of divisors of class $H-L_j$ where $H$ denotes the hyperplane class. 

\begin{lem}\label{lem:fermat-pencils}
For each $j=1,2,3$, the cubics cut out on $Q$ by the planes through $L_j$ form a base-point-free pencil whose general member is a smooth genus-one curve.
\end{lem}

\begin{proof}
Fix $j$ and let $l_0,l_1$ be two independent linear equations defining $L_j$. Since $L_j\subset Q$, the defining quartic equation of $Q$ can be written as $l_0A+l_1B=0$ where $A$ and $B$ are some cubic forms.

For $[a:b]\in\bP^1$, let $P_{[a:b]}$ be the plane
$a l_0+b l_1=0$. Its intersection with $Q$ is the union of $L_j$ and the residual cubic
\[
C_{[a:b]}=\{a l_0+b l_1=0,\ bA-aB=0\}.
\]

It is clear that the pencil has no base point outside $L_j$. If $p\in L_j$ were a base
point, then $p\in C_{[a:b]}$ for every $[a:b]\in\bP^1$, so
$bA(p)-aB(p)=0$ for every $[a:b]$. Consequently $A(p)=B(p)=0$. Since
$l_0(p)=l_1(p)=0$, we would then have
\[
d(l_0A+l_1B)_p=A(p)\,{dl_0}_p+B(p)\,{dl_1}_p=0,
\]
contrary to the smoothness of $Q$.

Thus the pencil is base-point-free. By Bertini's theorem, its general member is smooth. Each member is a plane cubic, so every smooth member has genus one.
\end{proof}

For $j=1,2,3$, let $\pi_j:Q\to\bP^1$ be the morphism defined by the
pencil of Lemma~\ref{lem:fermat-pencils}, and let $F_j$ denote its fiber divisor class. Since a plane through $L_j$ cuts $Q$ in $L_j$ and a fiber of $\pi_j$, we
have the equality $F_j=H-L_j$ between divisor classes on $Q$. As the intersection numbers satisfy $H^2=4$, $HL_j=1$, and as the lines $L_j$ are pairwise disjoint, we obtain $F_jF_k=2$ for $j\neq k$.

\begin{prop}\label{prop:fermat-model}
The morphism $\Phi:Q\to(\bP^1)^3$ induced by the three pencils constructed above is an isomorphism onto a smooth hypersurface of multidegree $(2,2,2)$.
\end{prop}

\begin{proof}
We first prove that $\Phi$ is a finite morphism. If an irreducible curve $C\subset Q$ were contracted by all three $\pi_j$, then for each $j$ it would lie in a plane through $L_j$. As the first two such planes inteserct in one line, the curve $C$ would itself be a line on $Q$. Remark that by construction this line would meet each $L_j$ because it shares a plane with each of them. A direct computation with the explicit equations of the three lines $L_j$ shows that no such line would exist. 

We next prove that $\Phi$ is birational onto its image, i.e.\ has degree one. We prove this by showing that a general value has only one preimage. At $p=[1:0:0:1]\in Q$, the three planes corresponding to the fibers of $\pi_1,\pi_2,\pi_3$ through $p$ have equations
\[
X-Y+Z-W=0,\quad
X-iY-Z-W=0,\quad
X+Y-iZ-W=0.
\]
The determinant of their coefficients in $X,Y,Z$ is $2$, so these three planes intersect in the single point $p$. The same is true for triples of
fiber planes parametrized by a nonempty Zariski-open subset of the image of $\Phi$. In other words a general fiber of $\Phi$ consists of one point.

Denote $Y=\Phi(Q)$. Since $\Phi$ is a finite morphism, $Y$ is an integral surface in $(\bP^1)^3$. As $(\bP^1)^3$ is smooth, the Weil divisor $Y$ is Cartier (see \cite[Prop.~2.6.11]{Hartshorne}). Let $H_j$ denote the class on $(\bP^1)^3$ of the pullback line bundle of $\cO_{\bP^1}(1)$ from the $j$-th factor of $(\bP^1)^3$. Then $[Y]=aH_1+bH_2+cH_3$ where $(a,b,c)$ is the multidegree of $Y$. Since $\Phi^*H_j=F_j$ and
$\Phi_*[Q]=[Y]$, we obtain the intersection numbers  
$a=F_2F_3$, $b=F_1F_3$, and $c=F_1F_2$ by the projection formula (see \cite[Prop.~2.3(c)]{FultonBook}). Hence $a=b=c=2$.

As $Y$ is a Cartier divisor in $(\bP^1)^3$, it is Cohen–-Macaulay, i.e.\ locally complete intersection (see \cite[Th.~II.8.21, Prop.~II.8.23]{Hartshorne}). By the adjunction formula (see \cite[Th.~7.11]{Hartshorne}), the dualizing sheaf of $Y$ is trivial. Since the dualizing sheaf of the smooth surface $Q$ is simply its canonical bundle, it is also trivial. Hence the conductor of the finite morphism $\Phi:Q\to Y$ is trivial by \cite[Prop.~2.9]{Piene} (see also \cite[Ex.~III.7.2]{Hartshorne}). This means that $\Phi$ is an isomorphism.
\end{proof}

\begin{cor}\label{cor:fermat-kummer}
The Fermat quartic is an Oka Kummer surface.
\end{cor}

\begin{proof}
By Proposition~\ref{prop:fermat-model} and Corollary~\ref{cor:222}, the Fermat quartic is Oka. It is a
Kummer surface by Mizukami's theorem, see \cite[App.~Th.~A.1]{ISZ}.
\end{proof}

\subsection{Deformations preserving the three pencils}\label{subsec:deformation}
Recall that $\Lambda$ is the K3 lattice from the introduction. Again we refer to \cite{Huybrechts} for background on lattices and cohomological aspects of K3 surfaces. 
In this subsection we will fix a marking $\coh^2(Q,\bZ)\to \Lambda$ which sends its Kummer sublattice to the fixed primitive sublattice $\Kum\subset\Lambda$. Via this marking, we also regard the fiber classes $F_1,F_2,F_3$ as elements of $\Lambda$. Let $\Lambda_Q\subset\Lambda$ be the smallest primitive sublattice
containing $\Kum,F_1,F_2,F_3$.

By Prop.~\ref{prop:fermat-model}, the class $F_1+F_2+F_3$ is the pullback of
$\cO_{(\bP^1)^3}(1,1,1)$, thus ample. The Hodge index theorem then implies that $\Lambda_Q$ is nondegenerate of signature
$(1,\operatorname{rk}\Lambda_Q-1)$. Since $\operatorname{rk}\Kum=16$, we have $\operatorname{rk}\Lambda_Q\leq 19$. Hence the period
domain 
\begin{equation}\label{eq:definition-period-Q}
\mathcal D_Q=\{P\in\mathcal D\mid P\perp \Lambda_Q\}
\end{equation}
is contained in the Kummer period domain 
$\mathcal D_K$ and has complex dimension
\begin{equation}\label{eq:period-Q-dimension}
20-\operatorname{rk}\Lambda_Q\geq 1.
\end{equation}

The following proposition performs a deformation of Prop.~\ref{prop:fermat-model}. It is the key geometric input for the dynamical machinery that we investigate later in Section~\ref{sec:periods}. 
\begin{prop}\label{prop:deformation-open}
A nonempty open subset of $\mathcal D_Q$ parametrizes Kummer surfaces which are smooth hypersurfaces of multidegree $(2,2,2)$ in $(\bP^1)^3$. All these surfaces are Oka.
\end{prop}

\begin{proof}
Let $p:\mathcal X\to T$ be a local Kuranishi family of $Q$ over a simply connected base $T$, with central point $0\in T$. Our fixed marking $\coh^2(Q,\bZ)\to\Lambda$ extends to a trivialization of $R^2p_*\bZ$, i.e.\ it induces markings on nearby fibers. The induced local period map
$\operatorname{per}:T\to\mathcal D$ is a local biholomorphism by
the local Torelli theorem (see \cite[Ch.~6,7]{Huybrechts}).

Let $B:=\operatorname{per}^{-1}(\mathcal D_Q)$. Up to shrinking $T$, we can and will assume that the period map sends $B$ biholomorphically to a neighborhood of the period of
$Q$ in $\mathcal D_Q$. By definition \eqref{eq:definition-period-Q} of $\mathcal D_Q$, every class in $\Lambda_Q$, and in particular each $F_i$, is orthogonal to the period, thus of type $(1,1)$ on every fiber over
$B$. Therefore they are classes of line bundles. The line bundles $\cO_Q(F_i)$ thus extend to line bundles $\mathcal L_i$ on $\mathcal X$ (see \cite[Ch.~6, \S2.4]{Huybrechts} for the arguments in this paragraph).

We have $\cO_Q(F_i)\simeq\pi_i^*\cO_{\bP^1}(1)$ for the fibration $\pi_i:Q\to\bP^1$, so $h^0(Q,\cO_Q(F_i))=2$. Moreover, $h^2(Q,\cO_Q(F_i))=h^0(Q,\cO_Q(-F_i))=0$ by Serre duality (see \cite[\S~I.5]{BHPV}). Since the self-intersection $F_i^2=0$ vanishes because $F_i$ is a fiber class, we get $\chi(Q,\cO_Q(F_i))=2$ by the Riemann--Roch theorem (see \cite[\S~I.5]{BHPV}). Therefore $h^1(Q,\cO_Q(F_i))=0$. 

For $b\in B$, let $X_b=p^{-1}(b)$ and
$\mathcal L_{i,b}=\mathcal L_i|_{X_b}$. Since
$h^1(Q,\cO_Q(F_i))=h^2(Q,\cO_Q(F_i))=0$, by using the semicontinuity theorem (see \cite[\S~7, Satz~3]{Grauert1960}) we can and will assume, up to shrinking $B$, that
$h^1(X_b,\mathcal L_{i,b})=h^2(X_b,\mathcal L_{i,b})=0$ for every
$b\in B$. The class of $\mathcal L_{i,b}$ has self-intersection $0$ by definition, so the Riemann--Roch theorem on the K3 surface $X_b$ implies
$\chi(X_b,\mathcal L_{i,b})=2$. Hence
$h^0(X_b,\mathcal L_{i,b})=2$ for every $b\in B$. 
By Grauert's theorem \cite[\S~7, Satz~5]{Grauert1960}, a basis of $\coh^0(Q,\cO_Q(F_i))$ extends to two sections of $\mathcal L_i$ whose restrictions form a basis of $\coh^0(X_b,\mathcal L_i|_{X_b})$ for every $b\in B$. As these two sections have no common zero on $Q$, up to shrinking $B$ we can and will assume that they have no common zero on any fiber. Therefore, by taking ratios of the two sections we obtain morphisms $\pi_{i,b}:X_b\to\bP^1$ which deform the three initial fibrations $\pi_i$.

The product
$\Phi_b=(\pi_{1,b},\pi_{2,b},\pi_{3,b}):X_b\to(\bP^1)^3$ is an embedding
for $b=0$ by Prop.~\ref{prop:fermat-model}. Since being an embedding is an open property, up to shrinking again $B$ we assert that every $\Phi_b$ is an embedding.
Its image is a smooth hypersurface which has necessarily multidegree $(2,2,2)$ (see the argument in Prop.~\ref{prop:fermat-model}). Thus every $X_b$ is Oka by Cor.~\ref{cor:222}.

Finally, remark that by construction $\Kum$ is primitive in the Néron--Severi group $\operatorname{NS}(X_b)$ of every fiber $X_b$. By Nikulin's criterion \cite[Th.~3]{NikulinKummer} we conclude that every $X_b$ is a Kummer surface.
\end{proof}

\begin{prop}\label{prop:rational-free-period}
There is a marked Oka Kummer surface with period $P_{\mathrm {irr}}\in\mathcal D_K$ such that
$P_{\mathrm {irr}}\cap\Lambda_\bQ=\{0\}$.
\end{prop}

\begin{proof}
Let $B\subset\mathcal D_Q$ be the nonempty open subset provided by Prop.~\ref{prop:deformation-open}. For a nonzero vector $v\in\Lambda_\bQ$, consider
the locus of periods $P\in\mathcal D_Q$ containing $v$. This locus is
empty unless $v\in ((\Lambda_Q)_\bR)^\perp$ and $v^2>0$. When it is nonempty, the
positive definite oriented two-planes containing the positive line $\bR v$ are parametrized by positive lines in
$v^\perp\cap ((\Lambda_Q)_\bR)^\perp$. They therefore form a closed real submanifold of dimension $20-\operatorname{rk}\Lambda_Q$, whereas $\mathcal D_Q$ has real
dimension $2(20-\operatorname{rk}\Lambda_Q)$. By \eqref{eq:period-Q-dimension} this locus is nowhere dense.

There are only countably many nonzero vectors in $\Lambda_\bQ$. Since
$B$ is a Baire space, it contains a period $P_{\mathrm {irr}}$ which contains none of
the above loci. In other words we can find a period satisfying $P_{\mathrm {irr}}\cap\Lambda_\bQ=\{0\}$.
\end{proof}

\section{Period domain and ergodic complex structure}\label{sec:periods}

Let $\Ort(\Lambda)$ be the isometry group of the K3 lattice $\Lambda$ and let $\Gamma\subset \Ort(\Lambda)$ be a finite-index subgroup preserving $\mathcal D$. By Verbitsky's theorem~\cite[Th.~2.5]{Verbitsky} which is an application of Ratner's theory~\cite{Ratner}, the lattice $P_{\mathrm {irr}}$ constructed in Prop.~\ref{prop:rational-free-period} has dense orbit in the period domain:
\begin{equation}\label{eq:dense-period-orbit}
\overline{\Gamma\cdot P_{\mathrm {irr}}}=\mathcal D.
\end{equation}
Every point of the orbit $\Gamma \cdot P_{\mathrm {irr}}$ is the period of a marking of the same Oka surface.

We need a similar density statement inside the Kummer period domain. Let $L=\Kum^\perp\subset\Lambda$ and let
$\Gamma_{\Kum}\subset \Ort(\Lambda)$ be the subgroup of $\Ort(L)$ acting trivially on
$\Kum$. We assert that the image of $\Gamma_{\Kum}$ in $\Ort(L)$ has finite index. Indeed, since $\Kum$ and $L$ are primitive orthogonal complements in the
unimodular lattice $\Lambda$, we have $\Kum\oplus L\subset\Lambda\subset\Kum^\vee\oplus L^\vee$ where $\Kum^\vee,L^\vee$ denote the dual lattices. If
$g\in \Ort(L)$ acts trivially on the discriminant group $L^\vee/L$, then $g(y)-y\in L$ for every $y\in L^\vee$, so $\operatorname{id}_{\Kum}\oplus g$ preserves
$\Lambda$. Thus the image of $\Gamma_{\Kum}$ contains
$\ker(\Ort(L)\to \Ort(L^\vee/L))$, which has finite index.

Let $\Gamma_{\Kum}^+$ be the subgroup of $\Gamma_{\Kum}$ preserving $\mathcal D_K$. Its
image has finite index in the subgroup of $\Ort(L)$ preserving
$\mathcal D_K$. Since $L$ has signature $(3,3)$ and
$P_{\mathrm{irr}}\cap L_\bQ=\{0\}$, Verbitsky's theorem \cite[Th.~2.5]{Verbitsky} still can be applied, and we obtain
\begin{equation}\label{eq:dense-kummer-orbit}
\overline{\Gamma_{\Kum}^+\cdot P_{\mathrm {irr}}}=\mathcal D_K.
\end{equation}
Every point of this orbit is again the period of a marking of the same Oka
Kummer surface.

\begin{proof}[Proof of Theorem~\ref{mainthm:period}]
By \eqref{eq:dense-period-orbit}, the set $\mathcal O$ is dense in the period domain $\mathcal D$. We claim that it is also $G_\delta$. By the local Torelli theorem~\cite[Ch.~6,7]{Huybrechts} every point of
$\mathcal D$ has a neighborhood biholomorphic to the base of a marked Kuranishi family. Two marked K3 surfaces with the same
period have biholomorphic underlying surfaces by the global Torelli theorem~\cite[Ch.~7, Th.~5.3]{Huybrechts}. Thus, the intersection of $\mathcal O$ with such a neighborhood coincides with the Oka locus of the corresponding Kuranishi family. This locus is $G_\delta$ by L\'arusson's theorem~\cite[Cor.~8]{Larusson}. Hence $\mathcal O$ is locally
$G_\delta$ in the period domain, and therefore $G_\delta$ by Lemma~\ref{lem:local-gdelta} below.

For every nonzero $h\in\Lambda$, the locus $\{P\in\mathcal D\mid P\perp h\}$ is a closed codimension-one analytic subset. Consequently the locus $\{P\in\mathcal D\mid\Lambda\cap P^\perp=0\}$ is a dense $G_\delta$ subset of $\mathcal{D}$. Its intersection with $\mathcal O$ is therefore again a dense $G_\delta$ subset because $\mathcal D$ is a Baire space. For every period in this intersection the corresponding K3 surface is nonprojective and contains no curve because its N\'eron--Severi group vanishes.

By \eqref{eq:dense-kummer-orbit}, $\mathcal O_K$ is dense in
$\mathcal D_K$, and it is $G_\delta$ because
$\mathcal O_K=\mathcal O\cap\mathcal D_K$. The same argument as above
shows that periods with N\'eron--Severi lattice exactly equal to $\Kum$ form a dense $G_\delta$ in $\mathcal D_K$. Indeed, if
$h\in\Lambda\setminus\Kum$, then its projection to $L_\bQ$ is nonzero because $\Kum$ is primitive. Hence the condition $P\perp h$ defines a closed hypersurface in
$\mathcal D_K$. Thus the Oka periods with N\'eron--Severi lattice exactly equal to $\Kum$ form a dense $G_\delta$ in $\mathcal D_K$. Since $\Kum$ is
negative definite, the corresponding K3 surfaces are nonprojective by the Hodge index theorem.
\end{proof}

\begin{lem}\label{lem:local-gdelta}
Let $M$ be a second-countable metrizable space and let $A\subset M$ be a subset. If $A$ is locally $G_\delta$, then it is $G_\delta$.
\end{lem}

\begin{proof}
The metrizability ensures that every open subset of $M$ is $F_\sigma$. Consider a countable open cover $(U_j)$ of $M$ such that $U_j\setminus A$ is $F_\sigma$ in $U_j$. Every subset closed in $U_j$ has the form $U_j\cap F$ with $F$ closed in $M$, hence is $F_\sigma$ in $M$ because $U_j$ is. Thus each $U_j\setminus A$ is $F_\sigma$ in $M$, and so is $M\setminus A=\bigcup_j(U_j\setminus A)$.
\end{proof}

\begin{proof}[Proof of Theorem~\ref{mainthm:K3-Kuranishi}]
Choose a marking of the local Kuranishi family and let
$\operatorname{per}:B\to\mathcal D$ be its period map. By the local Torelli theorem, $\operatorname{per}$ is a local biholomorphism.

Let $V\subset B$ be a nonempty open subset. Then
$\operatorname{per}(V)$ is open in $\mathcal D$. By Theorem~\ref{mainthm:period}, it contains a period represented by an Oka K3 surface. If $b\in V$ has this period, then the corresponding K3 fiber $X_b$ is biholomorphic to that surface by the global Torelli theorem. Therefore $V\cap B_{\mathrm{Oka}}\neq\emptyset$. Since $V$ was arbitrary, $B_{\mathrm{Oka}}$ is dense in $B$. It is $G_\delta$ by L\'arusson's theorem~\cite[Cor.~8]{Larusson}.

Similarly, Theorem~\ref{mainthm:period} implies that every
$\operatorname{per}(V)$ contains an Oka period $P$ satisfying
$\Lambda\cap P^\perp=0$. Since
$\varphi_b(\NS(X_b))=\Lambda\cap\operatorname{per}(b)^\perp$, the
corresponding fiber $X_b$ is Oka and satisfies $\NS(X_b)=0$. Thus these parameters are also dense in $B$. Moreover, the locus where $\NS(X_b)=0$ is the inverse image under $\operatorname{per}$ of
$\{P\in\mathcal D\mid\Lambda\cap P^\perp=0\}$, which is $G_\delta$ by the proof of Theorem~\ref{mainthm:period}. Its intersection with $B_{\mathrm{Oka}}$ is therefore $G_\delta$.
\end{proof}

\section{An Oka hyperk\"ahler fourfold}\label{sec:oka-fourfold}
An \emph{irreducible holomorphic symplectic manifold} is a simply connected compact K\"ahler manifold whose space of holomorphic $2$-forms is generated by a symplectic form. Such a manifold is also called an irreducible hyperk\"ahler manifold. We say that it is of \emph{$K3^{[n]}$-type} if it is deformation equivalent to the Hilbert scheme of length-$n$ subschemes of a K3 surface. For an irreducible hyperk\"ahler manifold there is a quadratic form, called Beauville--Bogomolov--Fujiki form, on the second integral cohomology (see~\cite{Beauville1983}). In this section, we only consider hyperk\"ahler fourfolds, and denote by $q$ the Beauville--Bogomolov--Fujiki form normalized by $\int\alpha^4=3q(\alpha)^2$.

\subsection{Construction}\label{subsec:hyperkahler-construction}
Let $S\subset(\bP^1)^3$ be a Wehler K3 surface, i.e.\ a smooth $(2,2,2)$ hypersurface. Its coordinate projections $p_i:S\to\bP^1$, for $i=1,2,3$, are genus-one fibrations. Let $F_i=c_1(p_i^*\cO_{\bP^1}(1))$ be their fiber cohomology classes. Let $X_0=S^{[2]}$ be the Hilbert scheme of length-two subschemes of $S$; it is a hyperkähler variety. Each $p_i$ induces a Lagrangian fibration $f_i^0:X_0\to\Sym^2(\bP^1)\simeq\bP^2$ where $\Sym^2$ denotes the symmetric product. A general fiber of $f_i^0$ is a product of two genus-one curves. 

Let $E\subset X_0$ be the exceptional divisor of the natural morphism $X_0=S^{[2]}\to\Sym^2(S)$, i.e.\ $E$ parametrizes nonreduced length-two subschemes of $S$. By~\cite[\S~6]{Beauville1983}, there is an orthogonal decomposition 
\begin{equation}\label{eq:orthogonal-decomposition}
\coh^2(X_0,\bZ)\simeq \Omega\oplus\bZ\delta
\end{equation}
where $2\delta=[E]$, $q(\delta)=-2$, and $q\vert_\Omega$ coincides with the intersection form on $S$. Since $\Omega$ has rank $22$, the second Betti number of $X_0$ is $b_2(X_0)=23$. Via the decomposition \eqref{eq:orthogonal-decomposition}, the class $F_i\in \Omega$ corresponds exactly to the fiber class $c_1((f_i^0)^*\cO_{\bP^2}(1))$. In the sequel we will denote both classes by $F_i$. 

Consider the sublattice $N:=\bZ F_1+\bZ F_2+\bZ F_3$. As $F_i^2=0$ and $F_i\cdot F_j=2$ for $i\neq j$, the intersection matrix on $N$ is
\begin{equation}\label{eq:fourfold-lattice}
 \bigl(q(F_i,F_j)\bigr)_{i,j=1}^3
 =\begin{pmatrix}0&2&2\\2&0&2\\2&2&0\end{pmatrix},
 \quad \text{so } N\simeq U(2)\oplus\langle-4\rangle,
\end{equation}
where $U(2)$ has Gram matrix $\begin{psmallmatrix}0&2\\2&0\end{psmallmatrix}$, and $\langle-4\rangle$ has a generator of square $-4$. Here the orthogonal decomposition in~\eqref{eq:fourfold-lattice} is compatible to the basis $F_1,F_2,F_3-F_1-F_2$. In particular, $N$ has signature $(1,2)$. By Lefschetz Theorem the restriction $\coh^2((\bP^1)^3,\bZ/p\bZ)\to\coh^2(S,\bZ/p\bZ)$ is injective for every prime number $p$. Therefore $N$ is a primitive sublattice in $\Omega$, and hence in $\coh^2(X_0,\bZ)$.

By \cite[\S~6.2]{MarkmanSurvey}, there exist a local Kuranishi family of the pair $(X_0, (f_i^0)^*\cO_{\bP^2}(1))$, i.e.\ a smooth proper holomorphic family $\pi:\cX\to T$ of $K3^{[2]}$-type fourfolds over a small contractible base, with $\pi^{-1}(0)=X_0$, and line bundles $\aL_i$ extending $(f_i^0)^*\cO_{\bP^2}(1)$. The fibers have markings $\coh^2(X_t,\bZ)\simeq\Omega\oplus\bZ\delta$ so that $c_1(\aL_i|_{X_t})=F_i$. The period map sends $T$ biholomorphically to an open neighborhood of $\coh^{2,0}(X_0)$ in the period domain
\begin{equation}\label{eq:fourfold-hodge-locus}
 \left\{[w]\in\bP(\coh^2(X_0,\bC))\mid q(w)=0,\ q(w,\overline w)>0,\
 q(w,F_i)=0\right\}.
\end{equation}
As $N$ has rank three, we have $\dim T=23-3-2=18$.

\begin{lem}\label{lem:fourfold-rank}
For $t$ sufficiently close to $0$, the line bundles $\aL_i\vert_{X_t}$ are globally generated and have three global sections. Their linear systems define Lagrangian fibrations onto $\bP^2$, and the product of these three fibrations has differential of rank at least two everywhere.
\end{lem}

\begin{proof}
By Matsushita's theorem~\cite[Th.~1.2 and Cor.~1.3]{MatsushitaDeformation}, the complete linear systems of $\aL_i|_{X_t}$ define holomorphic families of Lagrangian fibrations $f_i^t:X_t\to\bP^2$ extending $f_i^0$, for $t$ near $0$. In particular $\aL_i\vert_{X_t}$ is globally generated and has three global sections.

Consider a point $\xi\in X_0\setminus E$ representing two distinct points $x,y\in S$. For $u$ near $x$, we have $f_i^0(\{u,y\})=p_i(u)+p_i(y)\in \Sym^2(\bP^1)$. Note that each map $b\mapsto b+p_i(y)$ embeds $\bP^1$ into $\Sym^2(\bP^1)$ as a line. Since $(p_1,p_2,p_3):S\to (\bP^1)^3$ is an embedding, we deduce that $u\mapsto(f_1^0(\{u,y\}),f_2^0(\{u,y\}),f_3^0(\{u,y\}))$
is an immersion. Hence $d(f_1^0,f_2^0,f_3^0)_\xi$ has rank at least two.

Now we study what happens on $E$. Let $\rho:E\to S$ be the natural projection. For $\xi\in E$, we have $f_i^0(\xi)=2p_i(\rho(\xi))$. As $\rho$ is a submersion, $(p_1,p_2,p_3)$ is an embedding, and $b\mapsto2b$ embeds $\bP^1$ into $\Sym^2(\bP^1)$ as a conic, we deduce that the rank of $d(f_1^0,f_2^0,f_3^0)$ on $E$ is exactly two. 

Thus the product differential has rank at least two everywhere on $X_0$. The same holds for $d(f_1^t,f_2^t,f_3^t)$ for $t$ sufficiently close to $0$.
\end{proof}

For every class $\alpha\in \Omega\setminus N$, the equation $q(w,\alpha)=0$ defines a proper closed analytic subset of $T$. We may choose $t$ arbitrarily close to $0$, satisfying Lemma~\ref{lem:fourfold-rank} and avoiding all these countably many subsets. Let $X:=X_t$. Then
\begin{equation}\label{eq:fourfold-period-choice}
 \operatorname{NS}(X)=N.
\end{equation}
Let $f_i:X\to\bP^2$ be the Lagrangian fibrations defined by the complete linear systems of $\aL_i|_X$, and fix a holomorphic symplectic form $\sigma$ on $X$.

\begin{lem}\label{lem:kahler-positive-cone}
The hyperkähler fourfold $X$ is projective. Its positive cone is its Kähler cone. The class $F_1+F_2$ is ample while $(f_1,f_2):X\to (\bP^2)^2$ is a finite morphism.
\end{lem}
\begin{proof}
The positive cone of $X$ is the component of $\{\beta\in\coh^{1,1}(X,\bR)\mid q(\beta)>0\}$ containing the K\"ahler classes. As every element of $N$ has a square divisible by four, there is no integral $(1,1)$-class on $X$ of square $-2$ or $-10$ by~\eqref{eq:fourfold-period-choice}. Then by~\cite[Prop.~2.12]{Mongardi}, we deduce that the Kähler cone equals the positive cone. 

The class $F_1+F_2$ belongs to the positive cone and is a Kähler class. By Kodaira's embedding theorem, it is ample and $X$ is projective. Since a positive-dimensional fiber of $(f_1,f_2)$ would contain a curve on which the corresponding ample line bundle has degree zero, the map $(f_1,f_2):X\to\bP^2\times\bP^2$ is finite.
\end{proof}

\subsection{Automorphisms}
Let $N_\bR:=N\otimes_\bZ\bR$, and let $\Ort^+(N)$ be the group of integral isometries preserving the component of $\{\beta\in N_\bR\mid q(\beta)>0\}$ containing $F_1+F_2$. Denote by $N^\vee$ the dual lattice $\{u\in N\otimes_\bZ\bQ\mid q(u,N)\subset\bZ\}$. The discriminant group $N^\vee/N$ is finite, equipped with a quadratic form $[u]\mapsto q(u)\bmod 2\bZ$. We denote its isometry group by $\Ort(N^\vee/N)$, and the real algebraic group of determinant-one isometries of $N_\bR$ by $\SO(N_\bR,q)$.

Let $N^\perp$ be the orthogonal complement of $N$ in $\coh^2(X,\bZ)$, and let $G:=\{g\in\Aut(X)\mid g^*\vert_{N^\perp}=\operatorname{id}\}$. As the symplectic class satisfies $\sigma\in N^\perp\otimes\bC$, every element of $G$ preserves $\sigma$.

\begin{lem}\label{lem:fourfold-automorphisms}
The image of $G$ in $\Ort(N)$ contains the kernel of $\operatorname{O}^+(N)\to\operatorname{O}(N^\vee/N)$. Every finite-index subgroup of $G$ acts irreducibly on $N_\bR$.
\end{lem}

\begin{proof}
Let $\gamma\in\ker(\operatorname{O}^+(N)\to\operatorname{O}(N^\vee/N))$. Extend $\gamma$ to a rational isometry $\widetilde\gamma$ of $\coh^2(X,\bQ)$ by using the identity on $N^\perp\otimes\bQ$. For $\alpha\in\Omega$, we denote by $\alpha_N$ its orthogonal projection to $N\otimes\bQ$. Then $\alpha_N\in N^\vee$, so $\widetilde\gamma(\alpha)-\alpha=\gamma(\alpha_N)-\alpha_N\in N$. Hence $\widetilde\gamma(\Omega)=\Omega$. Since $\widetilde\gamma$ fixes the class $\delta$ (defined in \eqref{eq:orthogonal-decomposition}), it defines an integral isometry of $\coh^2(X,\bZ)$.

Note that $\widetilde\gamma$ fixes $\sigma$ and preserves the positive cone. Since $\Omega$ is unimodular and $q(\delta)=-2$, the discriminant group of $\coh^2(X,\bZ)$ is $\bZ/2\bZ$, generated by the class of $\delta/2$. We remak that $\widetilde\gamma$ acts trivially on the discriminant group of $\coh^2(X,\bZ)$. By~\cite[Def.~1.1 and Lem.~9.2]{MarkmanSurvey}, $\widetilde\gamma$ is a parallel-transport operator. Recall that by Lem.~\ref{lem:kahler-positive-cone} the K\"ahler cone equals the positive cone. Therefore we can apply the Global Torelli Theorem~\cite[Th.~1.3]{MarkmanSurvey} to find $g\in G$ such that $g^*=\widetilde\gamma$.

Recall that $F_1,F_2,(F_3-F_1-F_2)$ is a basis of $N$. We will think of an element $xF_1+yF_2+z(F_3-F_1-F_2)\in N$ as a symmetric matrix $M=\begin{psmallmatrix}x&z\\z&y\end{psmallmatrix}$. Then $q(M)=4\det M$, and $F_1+F_2$ corresponds to the identity matrix. Hence the action $M\mapsto AMA^{\mathsf t}$ defines a homomorphism $\SL_2(\bZ)\to\operatorname{O}^+(N)$.

Let $H\subset G$ be a subgroup of finite index. Then the image of $H$ has finite index in $\operatorname{O}^+(N)$. Its inverse image $\Gamma\subset\SL_2(\bZ)$ therefore has finite index and is Zariski-dense in $\SL_2(\bR)$. Since the representation of $\SL_2(\bR)$ on symmetric matrices is irreducible, we conclude that $H$ acts irreducibly on $N_\bR$.
\end{proof}

\subsection{Invariant surfaces}
For an irreducible subvariety $Z\subset X$, we denote by $Z_{\mathrm{reg}}$ its smooth locus, and define $\Stab_G(Z):=\{g\in G\mid g(Z)=Z\}$. 

\begin{lem}\label{lem:fourfold-invariant-surfaces}
If an irreducible surface $Z\subset X$ has finite orbit under $G$, then $\sigma\vert_{Z_{\mathrm{reg}}}=0$.
\end{lem}

\begin{proof}
Let $\mu:\widetilde Z\to X$ be the minimal resolution of singularities (see \cite[Th.~III.6.2]{BHPV}). The action of $\Stab_G(Z)$ lifts to $\widetilde Z$ by uniqueness of the resolution. Note that $\Stab_G(Z)$ has finite index in $G$ because $Z$ has finite orbit.

The map $\beta\mapsto\int_Z\beta^2$ defines a quadratic form on $N_\bR$ which is invariant under $\Stab_G(Z)$. By representation-irreducibility of $\Stab_G(Z)$ from Lemma~\ref{lem:fourfold-automorphisms}, there exists $c>0$ such that
\begin{equation}\label{eq:trapped-intersections}
 \int_{\widetilde Z}\mu^*\beta\,\mu^*\gamma =\int_Z\beta\gamma =c\,q(\beta,\gamma), \quad \beta,\gamma\in N_\bR.
\end{equation}

Suppose now for contradiction that $\mu^*\sigma\neq 0$. Then the kernel of $\mu^*:\coh^2(X,\bQ)\to\coh^2(\widetilde Z,\bQ)$ is a rational Hodge substructure without $(2,0)$-part or $(0,2)$-part because $\coh^{2,0}(X)=\bC\sigma$. By the Lefschetz $(1,1)$-theorem, it is contained in $\operatorname{NS}(X)\otimes\bQ=N\otimes\bQ$. If $\beta$ is in this kernel, then by~\eqref{eq:trapped-intersections} $q(\beta,\gamma)=0$ for every $\gamma\in N\otimes\bQ$. As $q$ is nondegenerate on $N$, we have $\beta=0$. Thus $\mu^*$ is injective.

Let $R$ be the zero divisor of $\mu^*\sigma$. Since $G$ preserves $\sigma$, the divisor $R$ is invariant under $\Stab_G(Z)$. By Lemma~\ref{lem:fourfold-automorphisms}, the invariant linear functional $\beta\mapsto R\cdot\mu^*\beta$ vanishes on $N_\bR$. For an ample class $\beta$, the degree of $\mu^*\beta$ on a curve is positive unless the curve is contracted by $\mu$. Therefore every component of $R$ is contracted by $\mu$.

For every $m\geq 1$ and nonzero $s\in\coh^0(\widetilde Z,mK_{\widetilde Z})$, the poles of $s/(\mu^*\sigma)^m$ are supported on $R$. Being a rational function on $Z^\nu$, $s/(\mu^*\sigma)^m$ has no poles in codimension one. It is thus constant. Therefore $h^0(\widetilde Z,mK_{\widetilde Z})=1$ for every $m\geq1$. By the classification of projective surfaces~\cite{BHPV}, there is a birational morphism $\rho:\widetilde Z\to Z_0$ where $Z_0$ is K3 or abelian. Hence the form $\mu^*\sigma$ descends to a nowhere-vanishing holomorphic $2$-form on $Z_0$. This implies that every $\rho$-exceptional curve occurs in $R$, thus is contracted by $\mu$. We then obtain
\begin{equation}\label{eq:fourfold-surface-cohomology}
 \mu^*\coh^2(X,\bQ)\subset\rho^*\coh^2(Z_0,\bQ).
\end{equation}
By the injectivity of $\mu^*$ and~\eqref{eq:fourfold-surface-cohomology}, we deduce $23=b_2(X)\leq b_2(Z_0)\leq22$, a contradiction.
\end{proof}

\subsection{The Oka property}
We will use the terminology of Section~\ref{sec:symplectic-sprays}. By Prop.~\ref{prop:finite-fibrations}, each $f_i$ induces an additive spray $s_i:f_i^*T_{\bP^2}^*\to X$ with vertical differential $A_i=\sigma^\sharp\circ df_i^*$. Let $\cP$ be the pseudogroup generated by $G$ and the local biholomorphisms induced these three sprays, as in \S~\ref{subsec:dynamical-criterion}. Define
\begin{equation}\label{eq:fourfold-defect}
D:=\left\{x\in X\mid \cap_{i=1}^3\ker(df_{i,x})\ne\{0\}\right\}.
\end{equation}

\begin{thm}\label{thm:oka-fourfold}
The projective hyperkähler fourfold $X$ is Oka. 
\end{thm}

\begin{proof}
The set $D$ is algebraic. Since $(f_1,f_2)$ is finite by Lem.~\ref{lem:kahler-positive-cone}, its differential has rank four on a nonempty Zariski-open subset, so $D\neq X$. The locus defined in~\eqref{eq:domination-locus} is $U=X\setminus D$. 

Suppose now for contradiction that a $\cP$-orbit $\cO$ is contained in $D$. Let $W:=\overline{\cO}^{\mathrm{Zar}}$. Then $W\subset D$ is a $G$-invariant algebraic subset.

Let us fix $i\in\{1,2,3\}$, a Zariski-open set $V\subset\bP^2$, and an algebraic one-form $a$ on $V$. On $f_i^{-1}(V)$, recall that in Section~\ref{sec:symplectic-sprays} we defined the vector field $V_{i,a}(x)=A_{i,x}(a(f_i(x)))$. As $\sigma$ and $f_i$ are algebraic, $V_{i,a}$ is algebraic. For a regular local function $\varphi$ vanishing on $W$ and a point $x\in\cO$ where $\varphi$ is defined, we have
\[
V_{i,a}(\varphi)(x)=
\left.\frac{d}{dt}\right\vert_{t=0}\varphi
\Bigl(s_i\bigl(x,t\,a(f_i(x))\bigr)\Bigr)=0
\]
because the local biholomorphisms in $\cP$ preserve $\cO$. By Zariski density, the local regular function $V_{i,a}(\varphi)$ vanishes on $W$ wherever it is defined. Let $Z$ be an irreducible component of $W$. Then $V_{i,a}$ is tangent to $Z$ at its smooth points outside the intersection locus with other components of $W$. The above arguments show that, at a general point $x\in Z_{\mathrm{reg}}$, we have
\begin{equation}\label{eq:fourfold-surface-tangency}
A_{i,x}\bigl(T^*_{f_i(x)}\bP^2\bigr)\subseteq T_xZ,\qquad i=1,2,3.
\end{equation}

If $\dim Z=0$, then $A_{i,x}=0$ for every $i$ by~\eqref{eq:fourfold-surface-tangency}. By~\eqref{eq:hamiltonian-differential}, we would have $df_{i,x}=0$ for every $i$, contradiction to Lemma~\ref{lem:fourfold-rank}.

If $\dim Z=3$, then $[Z]\in N$ would be fixed by $\Stab_G(Z)$ which is a finite-index subgroup of $G$ because $G$ permutes the components of $W$. This would contradict Lem.~\ref{lem:fourfold-automorphisms}. 

If $\dim Z=1$, then $\beta\mapsto\int_Z\beta$ would be a nonzero linear functional invariant under $\Stab_G(Z)$, again a contradiction to Lem.~\ref{lem:fourfold-automorphisms}. 

Therefore $\dim Z=2$, and by Lemma~\ref{lem:fourfold-invariant-surfaces} we have $\sigma\vert_{Z_{\mathrm{reg}}}=0$. 
For $i=1,2$, $v\in T_xZ$ and $\lambda\in T^*_{f_i(x)}\bP^2$, by~\eqref{eq:hamiltonian-differential} and~\eqref{eq:fourfold-surface-tangency} we have
\[
\lambda(df_{i,x}(v)) =\sigma_x(A_{i,x}(\lambda),v)=0.
\]
Hence both $f_1$ and $f_2$ are constant on $Z$. This contradicts the finiteness of $(f_1,f_2)$ proved in Lem.~\ref{lem:kahler-positive-cone}.

Every $\cP$-orbit therefore meets $U$, and we conclude by Prop.~\ref{prop:orbit}.
\end{proof}

As our fourfold $X$ was constructed by choosing a Baire generic parameter $t$ (see the end of \S~\ref{subsec:hyperkahler-construction}), we proved:
\begin{cor}
There are infinitely many Oka projective hyperkähler manifolds.
\end{cor}

\section*{Acknowledgements}
We are grateful to Yuta Kusakabe for his valuable comments and suggestions. S.-Y. Xie thanks Steven Lu for bringing Verbitsky's work on ergodic complex structures to his attention during his visit to the Institute of Mathematics, Academia Sinica, in June 2026. He also thanks Julie Tzu-Yueh Wang and the Institute for their hospitality. S.-Y. Zhao thanks Serge Cantat for teaching him years ago the theory of K3 surfaces and Luanyin Shen for kind help with Grauert's theorems.

\section*{Funding}

S.-Y. Xie acknowledges partial support from the National Key R\&D Program of China under Grants No. 2023YFA1010500 and No. 2021YFA1003100, and from the National Natural Science Foundation of China under Grants No. 12288201 and No. 12471081, as well as support from the  Xiaomi Young Talents Program.

S.-Y. Zhao acknowledges partial support from the French National Research Agency under the projects GAG (ANR-24-CE40-3526-01) and DynAtrois (ANR-24-CE40-1163).

\section*{AI usage}
The authors developed the core materials on K3 surfaces for Th.~\ref{mainthm:K3-Kuranishi}. The hyperkähler example is found through interations of the authors with LLM models. LLM models have also been used for the editorial work of the manuscript. The authors take full responsability of the paper.

\bibliographystyle{amsplain}
\bibliography{biblio}

\end{document}